\documentclass{amsart}%
\usepackage{amsfonts}
\usepackage{amsmath, amsthm}
\usepackage{amssymb}
\usepackage[dvips]{graphicx}
\usepackage{amsmath}%
\providecommand{\U}[1]{\protect\rule{.1in}{.1in}}
\newtheorem{theorem}{Theorem}
\theoremstyle{plain}

\newtheorem{definition}{Definition}
\newtheorem{example}{Example}

\newtheorem{lemma}{Lemma}

\newtheorem{remark}{Remark}

\numberwithin{equation}{section}
\begin{document}
\title[Relative Convexity and Its Applications]{Relative Convexity and Its Applications}
\author{Constantin P. Niculescu}
\address{Department of Mathematics, University of Craiova, Craiova 200585, Romania}
\email{cpniculescu@gmail.com}
\author{Ionel Roven\c{t}a}
\address{Department of Mathematics, University of Craiova, Craiova 200585, Romania}
\email{ionelroventa@yahoo.com}
\thanks{Presented to the \emph{Workshop on Convex Functions and Inequalities},
Targoviste, June 14, 2014}
\date{October 1, 2014}

\begin{abstract}
We discuss a rather general condition under which the inequality of Jensen
works for certain convex combinations of points not all in the domain of
convexity of the function under attention. Based on this fact, an extension of
the Hardy-Littlewood-P\'{o}lya theorem of majorization is proved and new
insight is given into the problem of risk aversion in mathematical finance.

\end{abstract}
\maketitle

\section{Introduction}

The important role played by the classical inequality of Jensen in
mathematics, probability theory, economics, statistical physics, information
theory etc. is well known. See the books of Niculescu and Persson
\cite{NP2006}, Pe\v{c}ari\'{c}, Proschan and Tong \cite{PPT} and Simon
\cite{Simon}. The aim of this paper is to discuss a rather general condition
under which the inequality of Jensen works in a framework that includes a
large variety of nonconvex functions and to provide on this basis applications
to majorization theory and mathematical finance.

The possibility to extend the inequality of Jensen outside the framework of
convex functions was first noticed twenty years ago by Dragomirescu and Ivan
\cite{DI1993}. Later, Pearce and Pe\v{c}ari\'{c} \cite{PP1997} and Czinder and
P\'{a}les \cite{CP} have considered the special case of mixed convexity
(assuming the symmetry of the graph with respect to the inflection point). For
related results, see the papers of Florea and Niculescu \cite{FN2007},
Niculescu and Spiridon \cite{NSp2013}, and Mihai and Niculescu \cite{MN2015}.

The inequality of Jensen characterizes the behavior of a continuous convex
function with respect to a mass distribution on its domain. More precisely, if
$f$ is a continuous convex function on a compact convex subset $K$ of
$\mathbb{R}^{N}$ and $\mu$ is a Borel probability measure on $K$ having the
barycenter
\[
b_{\mu}=\int_{K}xd\mu(x),
\]
then the value of $f$ at $b_{\mu}$ does not exceed the mean value of $f$ over
$K,$ that is,%
\[
f(b_{\mu})\leq\int_{K}f(x)d\mu(x).
\]

A moment's reflection reveals that the essence of this inequality is the fact
that $b_{\mu}$ is a point of convexity of $f$ relative to its domain $K$. The
precise meaning of the notion of point of convexity is given in Definition
\ref{def1} below, which is stated in the framework of real-valued continuous
functions $f$ defined on a compact convex subset $K$ of $\mathbb{R}^{N}$.

\begin{definition}
\label{def1}A point $a\in K$ is a point of convexity of the function $f$
relative to the convex subset $V$ of $K$\emph{ }if $a\in V$ and
\begin{equation}
f(a)\leq\sum_{k=1}^{n}\lambda_{k}f(x_{k}), \tag{$J$}\label{J}%
\end{equation}
for every family of points $x_{1},...,x_{n}$ in $V$ and every family of
positive weights $\lambda_{1},...,\lambda_{n}$ with $\sum_{k=1}^{n}\lambda
_{k}=1$ and $\sum_{k=1}^{n}\lambda_{k}x_{k}=a.$\newline

The point $a$ is a point of concavity if it is a point of convexity for $-f$
\emph{(}equivalently, if the inequality $(J)$ works in the reversed
way\emph{)}.
\end{definition}

In what follows, the set $V$ that appears in Definition 1 will be referred to
as a\ \emph{neighborhood of convexity} of $a.$ Here, the term of
\emph{neighborhood} has an extended meaning and is not necessarily ascribed to
the topology of $\mathbb{R}^{N}.$ In order to avoid the trivial case where
$V=\left\{  a\right\}  ,$ we will always assume that $V$ is an infinite set;
for example, this happens when $a$ belongs to the relative interior of $V$
(the interior within the affine hull of $V)$.

For the function $f(x,y)=x^{2}-y^{2}$, $\ $the origin is a point of convexity
relative to the $Ox$ axis, and a point of concavity relative to the $Oy$ axis.
With respect to the plane topology, both axes have empty interior.

If a function\thinspace$f:K\rightarrow\mathbb{R}$ is convex, then every point
of $K$ is a point of convexity relative to the whole domain $K$ (and this fact
characterizes the property of convexity of $f$).

Definition 1 is motivated mainly by the existence of nonconvex functions that
admit points of convexity relative to the whole domain (or at least to a
neighborhood of convexity where the function is not convex). An illustration
is offered by the nonconvex function $g(x)=\left\vert x^{2}-1\right\vert ;$
for it, all points in $(-\infty,-1]\cup\lbrack1,\infty)$ are points of
convexity relative to the entire real set $\mathbb{R}$.

Every point of local minimum of a continuous function $f:[0,1]\rightarrow
\mathbb{R}$ is a point of convexity. Thus, every nowhere differentiable
continuous function $f:[0,1]\rightarrow\mathbb{R}$ admits points of convexity
despite the fact that it is not convex on any nondegenerate interval.

The idea of point of convexity is not entirely new. In an equivalent form, it
is present in the paper of Dragomirescu and Ivan \cite{DI1993}. The technique
of convex minorants, described by Steele \cite{Steele} at pp. 96 - 99, is also
closed to the concept of point of convexity.

A different concept of \emph{punctual convexity} is discussed in the recent
paper of Florea and P\u{a}lt\u{a}nea \cite{FP2014}.

\section{The Existence of Points of Convexity}

The following lemma indicates a simple geometric condition under which a point
is a point of convexity relative to the whole domain.

\begin{lemma}
\label{lem1}Assume that $f$ is a real-valued continuous function defined on a
compact convex subset $K$ of $\mathbb{R}^{N}.$ If $f$ admits a supporting
hyperplane at a point $a,$ then $a$ is point of convexity of $f$ relative to
$K.$

In other words, every point at which the subdifferential is nonempty is a
point of convexity.
\end{lemma}

\begin{proof}
Indeed, the existence of a supporting hyperplane at $a$ is equivalent to the
existence of an affine function $h(x)=\langle x,v\rangle+c$ (for suitable
$v\in\mathbb{R}^{N}$ and $c\in\mathbb{R}$) such that%
\[
f(a)=h(a)\text{ and }f(x)\geq h(x)\text{ for all }x\in K.
\]
If $\mu$ is a Borel probability measure, its barycenter is given by the
formula%
\[
b_{\mu}=\int_{K}xd\mu(x),
\]
so that if $b_{\mu}=a,$ then
\[
f(a)=h(a)=h\left(  \int_{K}xd\mu(x)\right)  =\int_{K}h(x)d\mu(x)\leq\int
_{K}f(x)d\mu(x).
\]

\end{proof}

\begin{remark}
Another sufficient condition for the convexity at a point, formulated in terms
of secant line slopes, can be found in the papers of Niculescu and Stephan
\cite{NSt2012}, \cite{NSt2013}. However, as shows the case of polynomials of
fourth degree, that condition does not overcome the result of Lemma 1.
\end{remark}

As is well known, the usual property of convexity assures the existence of a
supporting hyperplane at each interior point. See \cite{NP2006}, Theorem
3.7.1, p. 128. This explains why Jensen's inequality works nicely in the
context of continuous convex functions.

In the case of differentiable functions, the supporting hyperplane is unique
and coincides with the tangent hyperplane. For such functions, Lemma 1 asserts
that every point where the tangent hyperplane lies above/below the graph is a
point of concavity/convexity.

\begin{example}
In the one real variable case, the existence of points of convexity of a
nonconvex differentiable function (such as $xe^{x},$ $x^{2}e^{-x},$ $\log
^{2}x,$ $\frac{\log x}{x}$ etc.) can be easily proved by looking at the
position of the tangent line with respect to the graph.

For example, the function $xe^{x}$ is concave on $(-\infty,-2]$ and convex on
$[-2,\infty)$ (and attains a global minimum at $x=-1).$ See Figure 1.%
\begin{figure}
[h]
\begin{center}
\includegraphics[
height=3.9246cm,
width=5.8788cm
]%
{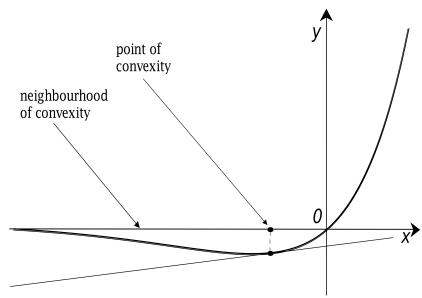}%
\caption{A point of convexity of the function $xe^{x}$ relative to the whole
real axis.}%
\label{Fig1}%
\end{center}
\end{figure}

Based on Lemma 1, one can easily show that every point $x\geq-1$ is a point of
convexity of $f$ relative to the whole real axis. Therefore%
\[
\sum_{k=1}^{n}\lambda_{k}x_{k}e^{x_{k}}\geq\left(  \sum_{k=1}^{n}\lambda
_{k}x_{k}\right)  e^{\sum_{k=1}^{n}\lambda_{k}x_{k}},
\]
whenever $\sum_{k=1}^{n}\lambda_{k}x_{k}\geq-1.$

In the special case where $\sum_{k=1}^{n}\lambda_{k}x_{k}\geq0,$ this
inequality can be deduced from Chebyshev's inequality and the convexity of the
exponential function. Borwein and Girgensohn \cite{BG} proved that
\[
\sum_{k=1}^{n}x_{k}e^{x_{k}}\geq\frac{\max\left\{  2,e(1-1/n)\right\}  }%
{n}\sum_{k=1}^{n}x_{k}^{2},
\]
for every family of real numbers $x_{1},x_{2},...,,x_{n}$ such that
$\sum_{k=1}^{n}x_{k}\geq0$. The extension of their result to the weighted case
$($subject to the condition $\sum_{k=1}^{n}\lambda_{k}x_{k}\geq0)$ is an open problem.
\end{example}

\begin{example}
The two real variables function
\[
f(x,y)=e^{-x^{2}-y^{2}},\quad\left(  x,y\right)  \in\mathbb{R}^{2},
\]
exhibits the phenomenon of relative concavity. Indeed, its graph is the
rotation graph of the function $z=e^{-x^{2}}$ around the $Oz$ axis and this
makes possible to apply Lemma 1 by means of calculus of one real variable. See
Figure 2.%
\begin{figure}
[h]
\begin{center}
\includegraphics[
height=4.0863cm,
width=6.1224cm
]%
{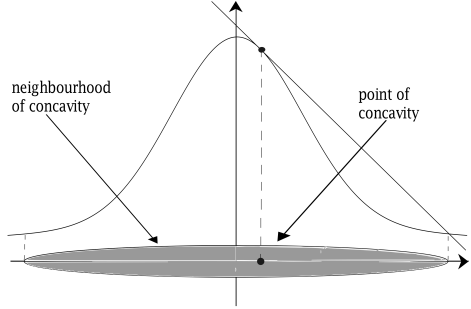}%
\caption{A point of concavity of the function $e^{-x^{2}-y^{2}}$ and a
neighborhood of concavity of it.}%
\label{Fig2}%
\end{center}
\end{figure}

The convexity properties of the function $f$ can be described in a more
convenient way by viewed it as a function of complex variable, via the formula
$f(w)=e^{-\left\vert w\right\vert ^{2}}.$ The function $f$ is strictly concave
on the compact disc $\overline{D}_{1/\sqrt{2}}\left(  0\right)  $ and attains
a global maximum at the origin.The tangent plane at the graph of $f,$ at any
point $w_{0}=(x_{0},y_{0})$ with $\left\Vert w_{0}\right\Vert \leq1/2,$ is
above the graph over a neighborhood of $w_{0}$ including the closed disc
$\overline{D}_{r^{\ast}}\left(  0\right)  $, where $r^{\ast}=1.\,\allowbreak
183\,802...$ is the solution of the equation $e^{-1/4}(\frac{3}{2}%
-x)=e^{-x^{2}}$. As a consequence,%
\[
\sum_{k=1}^{n}\lambda_{k}e^{-\left\vert w_{k}\right\vert ^{2}}\leq e^{-M^{2}}%
\]
for all points $w_{1},...,w_{n}\in\overline{D}_{r^{\ast}}\left(  0\right)  $
and all $\lambda_{1},...,\lambda_{n}>0$ such that $\sum_{k=1}^{n}\lambda
_{k}=1$ and $\left\vert \sum\limits_{k=1}^{n}\lambda_{k}w_{k}\right\vert
=M\leq1/2.$ Notice that Jensen's inequality yields this conclusion only when
$w_{1},...,w_{n}\in\overline{D}_{1/\sqrt{2}}\left(  0\right)  .$
\end{example}

The real variable case has also nontrivial implications in the case of matrix
functions. The function $F(X)=\operatorname{trace}(f(X))$ is convex/concave on
the linear space $\operatorname*{Sym}(n,\mathbb{R}),$ of all self-adjoint
(that is, symmetric) matrices in $\operatorname{M}_{n}(\mathbb{R}),$ whenever
$f:\mathbb{R}\rightarrow\mathbb{R}$ is convex/concave. See the paper of Lieb
and Pedersen \cite{LP} for details. Thus, in the case of $f(x)=xe^{x},$ the
function $F$ is concave on the convex set $\operatorname*{Sym}_{sp\subset
(-\infty,-2]}(n,\mathbb{R}),$ of all symmetric matrices in $\operatorname{M}%
_{n}(\mathbb{R})$ whose spectrum is included in $(-\infty,-2];$ the function
$F$ is convex on the set $\operatorname*{Sym}_{sp\subset\lbrack-2,\infty
)}(n,\mathbb{R}),$ of all symmetric matrices in $\operatorname{M}%
_{n}(\mathbb{R})$ whose spectrum is included in $(-2,\infty].$

The following result is a direct consequence of functional calculus with
self-adjoint matrices. If $\lambda_{1},...,\lambda_{n}$ are positive numbers
such that $\sum_{k=1}^{n}\lambda_{k}=1$ and $A_{1},...,A_{n}$ are matrices in
$\operatorname*{Sym}_{sp\subset(-\infty,-2]}(n,\mathbb{R})\cup
\operatorname*{Sym}_{sp\subset\lbrack-2,\infty)}(n,\mathbb{R})$ such that
$\sum_{k=1}^{n}\lambda_{k}A_{k}\geq-I_{n},$ then%
\[
\sum_{k=1}^{n}\lambda_{k}\operatorname{trace}\left(  A_{k}e^{A_{k}}\right)
\geq\operatorname{trace}\left[  \left(  \sum_{k=1}^{n}\lambda_{k}A_{k}\right)
e^{\sum_{k=1}^{n}\lambda_{k}A_{k}}\right]  .
\]

\section{The Extension of Hardy-Littlewood-P\'{o}lya theorem of majorization}

The notion of point of convexity leads to a very large generalization of the
Hardy-Littlewood-P\'{o}lya theorem of majorization. Given a vector
$\mathbf{x}=(x_{1},...,x_{N})$ in $\mathbb{R}^{N},$ let $\mathbf{x}%
^{\downarrow}$ be the vector with the same entries as $\mathbf{x}$ but
rearranged in decreasing order,%
\[
x_{1}^{\downarrow}\geq\cdots\geq x_{N}^{\downarrow}.
\]
The vector $\mathbf{x}$ is \emph{majorized} by $\mathbf{y}$ (abbreviated,
$\mathbf{x}\prec\mathbf{y})$ if%
\[
\sum_{i\,=\,1}^{k}\,x_{i}^{\downarrow}\leq\sum_{i\,=\,1}^{k}\,y_{i}%
^{\downarrow}\quad\text{for }k=1,...,N-1
\]
and%
\[
\sum_{i\,=\,1}^{N}\,x_{i}^{\downarrow}=\sum_{i\,=\,1}^{N}\,y_{i}^{\downarrow
}\,.
\]

The concept of majorization admits an order-free characterization based on the
notion of doubly stochastic matrix. Recall that a matrix $A\in
\,\operatorname{M}_{n}(\mathbb{R})$ is \emph{doubly stochastic} if it has
nonnegative entries and each row and each column sums to unity.

\begin{theorem}
\label{ThmHLP}\emph{(Hardy, Littlewood and P\'{o}lya \cite{HLP}).} Let
$\mathbf{x}$ and $\mathbf{y}$ be two vectors in $\mathbb{R}^{N}$, whose
entries belong to an interval $I.$ Then the following statements are equivalent:

$a)$ $\mathbf{x}\prec\mathbf{y};$

$b)$ there is a doubly stochastic matrix $A=(a_{ij})_{1\leq i,j\leq N}$ such
that $\mathbf{x}=A\mathbf{y};$

$c)$ the inequality $\sum_{i=1}^{N}f(x_{i})\leq\sum_{i=1}^{N}f(y_{i})$ holds
for every continuous convex function $f:I\rightarrow\mathbb{R}$.
\end{theorem}

An alternative characterization of the relation of majorization is given by
the Schur-Horn theorem: $\mathbf{x}\prec\mathbf{y}$ in $\mathbb{R}^{N}$ if and
only if the components of $\mathbf{x}$ and $\mathbf{y}$ are respectively the
diagonal elements and the eigenvalues of a self-adjoint matrix. The details
can be found in the book of Marshall, Olkin and Arnold \cite{MOA}, pp. 300-302.

The notion of majorization is generalized by weighted majorization, that
refers to probability measures rather than vectors. This is done by
identifying any vector $\mathbf{x}=(x_{1},...,x_{N})$ in $\mathbb{R}^{N}$ with
the probability measure $\frac{1}{N}\sum_{i=1}^{N}\delta_{x_{i}},$ where
$\delta_{x_{i}}$ denotes the Dirac measure concentrated at $x_{i}.$

We define the relation of majorization%
\begin{equation}
\sum_{i=1}^{m}\lambda_{i}\delta_{\mathbf{x}_{i}}\prec\sum_{j=1}^{n}\mu
_{j}\delta_{\mathbf{y}_{j}}, \tag{$2$}\label{2}%
\end{equation}
between two positive discrete measures supported at points in $\mathbb{R}^{N}%
$, by asking the existence of a $m\times n$-dimensional matrix $A=(a_{ij}%
)_{i,j}$ such that%
\begin{gather}
a_{ij}\geq0,\text{ for all }i,j\tag{$3$}\label{3}\\
\sum_{j=1}^{n}a_{ij}=1,\text{\quad}i=1,...,m\tag{$4$}\label{4}\\
\mu_{j}=\sum_{i=1}^{m}a_{ij}\lambda_{i},\text{\quad}j=1,...,n \tag{$%
5$}\label{5}%
\end{gather}
and%
\begin{equation}
\mathbf{x}_{i}=\sum_{j=1}^{n}a_{ij}\mathbf{y}_{j}\text{,\quad}i=1,...,m.
\tag{$6$}\label{6}%
\end{equation}

The matrices verifying the conditions $(3)\&(4)$ are called \emph{stochastic
on rows}. When $m=n$ and all weights $\lambda_{i}$ and $\mu_{j}$ are equal to
each others, the condition $(5)$ assures the \emph{stochasticity on columns,
}so in that case we deal with doubly stochastic matrices.

We are now in a position to state the following generalization of the
Hardy-Littlewood-P\'{o}lya theorem of majorization:

\begin{theorem}
\label{ThmGHLP}Suppose that $f$ is a real-valued function defined on a compact
convex subset $K$ of $\mathbb{R}^{N}$ and $\sum_{i=1}^{m}\lambda_{i}%
\delta_{\mathbf{x}_{i}}$ and $\sum_{j=1}^{n}\mu_{j}\delta_{\mathbf{y}_{j}}$
are two positive discrete measures concentrated at points in $K.$ If
$\mathbf{x}_{1},...,\mathbf{x}_{m}$ are points of convexity of $f$ relative to
$K$ and%
\[
\sum_{i=1}^{m}\lambda_{i}\delta_{\mathbf{x}_{i}}\prec\sum_{j=1}^{n}\mu
_{j}\delta_{\mathbf{y}_{j}},
\]
then
\begin{equation}
\sum_{i=1}^{m}\lambda_{i}f(\mathbf{x}_{i})\leq\sum_{j=1}^{n}\mu_{j}%
f(\mathbf{y}_{j}). \tag{$7$}\label{7}%
\end{equation}

\end{theorem}

\begin{proof}
By our hypothesis, there exists a $m\times n$-dimensional matrix
$A=(a_{ij})_{i,j}$ that is stochastic on rows and verifies the conditions
$(5)$ and $(6)$. The last condition shows that each point $\mathbf{x}_{i}$ is
the barycenter of the probability measure $\sum_{j=1}^{n}a_{ij}\delta
_{\mathbf{y}_{j}}$. By Jensen's inequality, we infer that
\[
f(\mathbf{x}_{i})\leq\sum_{j=1}^{n}a_{ij}f(\mathbf{y}_{j}).
\]
Multiplying each side by $\lambda_{i}$ and then summing up over $i$ from $1$
to $m,$ we conclude that
\[
\sum_{i=1}^{m}\lambda_{i}f(\mathbf{x}_{i})\leq\sum_{i=1}^{m}\left(
\lambda_{i}\sum_{j=1}^{n}a_{ij}f(\mathbf{y}_{j})\right)  =\sum_{j=1}%
^{n}\left(  \sum_{i=1}^{m}a_{ij}\lambda_{i}\right)  f(\mathbf{y}_{j}%
)=\sum_{j=1}^{n}\mu_{j}f(\mathbf{y}_{j}),
\]
and the proof of $(7)$ is done.
\end{proof}

\begin{example}
The well known Gauss--Lucas theorem on the distribution of the critical points
of a polynomial asserts that the roots $(\mu_{k})_{k=1}^{n-1}$ of the
derivative $P^{\prime}$ of any complex polynomial $P\in$ $\mathbb{C}[z]$ of
degree $n\geq2$ lie in the smallest convex polygon containing the roots
$(\lambda_{j})_{j=1}^{n}$ of the polynomial $P$. This led Malamud \cite{M2005}
to the interesting remark that the two families of roots are actually related
by the relation of majorization. Based on this remark, he was able to prove
the following conjecture raised by de Bruijn and Springer in 1947: for any
convex function $f:\mathbb{C}\rightarrow\mathbb{R}$ and any polynomial $P$ of
degree $n\geq2,$
\[
\frac{1}{n-1}\sum_{k=1}^{n-1}f(\mu_{k})\leq\frac{1}{n}\sum_{j=1}^{n}%
f(\lambda_{j}),
\]
where $(\lambda_{j})_{j=1}^{n}$ and $(\mu_{k})_{k=1}^{n-1}$ are respectively
the roots of $P$ and $P^{\prime}.$

Theorem \ref{ThmGHLP} allows us to relax the condition of convexity by asking
only that all the roots $\mu_{k}$ of $P^{\prime}$ be points of convexity for
$f.$ According to a remark above concerning the function $e^{-\left\vert
z\right\vert ^{2}},$ this implies that
\[
\frac{1}{n-1}\sum_{k=1}^{n-1}e^{-|\mu_{k}|^{2}}\geq\frac{1}{n}\sum_{j=1}%
^{n}e^{-\left\vert \lambda_{j}\right\vert ^{2}},
\]
whenever the roots $\mu_{1},...,\mu_{n-1}$ belong to $\overline{D}%
_{1/2}\left(  0\right)  $ and $\lambda_{1},...,\lambda_{n}$ belong to
$\overline{D}_{1.18}\left(  0\right)  .$ An example of polynomial verifying
these conditions is $P(z)=4z^{3}-3z.$
\end{example}

\begin{example}
A second application of Theorem \ref{ThmGHLP} refers to the function
$f(x)=\log^{2}x.$ This function is convex on the interval $(0,e]$ and concave
on $[e,\infty)$. The Hardy-Littlewood-P\'{o}lya theorem of majorization easily
yields the implication
\begin{equation}
\left(  x_{1},...,x_{n}\right)  \prec\left(  y_{1},...,y_{n}\right)
\Rightarrow\sum_{i=1}^{n}\log^{2}x_{i}\leq\sum_{i=1}^{n}\log^{2}y_{i}
\tag{$8$}\label{8}%
\end{equation}
whenever $x_{1},...,x_{n}$ and $y_{1},...,y_{n}$ belong to $(0,e].$ According
to Lemma 1, all points in $(0,2]$ are points of convexity of $f$ relative to
$(0,a^{\ast}],$ where%
\[
a^{\ast}=5.495\,869\,874...
\]
is the solution of the equation $\log^{2}x-\log^{2}2=\left(  \log2\right)
(x-2).$ By Theorem \ref{ThmGHLP}, the implication $(8)$ still works when
$x_{1},...,x_{n}\in(0,2]$ and $y_{1},...,y_{n}\in(0,a^{\ast}].$ Recently,
B\^{\i}rsan, Neff and Lankeit \cite{BNL} noticed still another case where an
inequality of the form $(8)$ holds true. Precisely, they proved that for every
two triplets $x_{1},x_{2},x_{3}$ and $y_{1},y_{2},y_{3}$ of positive numbers
which verify the conditions%
\[
x_{1}+x_{2}+x_{3}\leq y_{1}+y_{2}+y_{3},\text{\quad}x_{1}x_{2}+x_{2}%
x_{3}+x_{3}x_{1}\leq y_{1}y_{2}+y_{2}y_{3}+y_{3}y_{1}\text{ }%
\]
and $x_{1}x_{2}x_{3}=y_{1}y_{2}y_{3},$ we have%
\[
\sum_{i=1}^{3}\log^{2}x_{i}\leq\sum_{i=1}^{3}\log^{2}y_{i}.
\]

This suggests a new concept of majorization for $n$-tuples of positive
elements, based on elementary symmetric functions. Being beyond the purpose of
this paper, we will not enter the details.
\end{example}

Theorem \ref{ThmGHLP} provides the following extension of Popoviciu's inequality:

\begin{theorem}
\label{TPOP} Suppose that $f$ is a real-valued function defined on an interval
$I$. If $a,b,c$ belong to $I$ and $\frac{a+b}{2},\frac{a+c}{2}$ and
$\frac{b+c}{2}$ are points of convexity of $f$ relative to the entire interval
$I,$ then%
\begin{multline}
\frac{f\left(  a\right)  +f\left(  b\right)  +f\left(  c\right)  }{3}+f\left(
\frac{a+b+c}{3}\right) \tag{$9$}\label{9}\\
\geq\frac{2}{3}\left[  f\left(  \frac{a+b}{2}\right)  +f\left(  \frac{a+c}%
{2}\right)  +f\left(  \frac{b+c}{2}\right)  \right]  .\nonumber
\end{multline}

\end{theorem}

\begin{proof}
Without loss of generality we may assume that $a\geq b\geq c$. Then
\[
\frac{a+b}{2}\geq\frac{a+c}{2}\geq\frac{b+c}{2}\text{ and }a\geq\frac
{a+b+c}{3}\geq c.
\]

We attach to the points $a,b,c$ two sextic families of points:
\begin{align*}
x_{1}  &  =x_{2}=\frac{a+b}{2},\;x_{3}=x_{4}=\frac{a+c}{2},\;x_{5}=x_{6}%
=\frac{b+c}{2}\\
y_{1}  &  =a,\;y_{2}=y_{3}=y_{4}=\frac{a+b+c}{3},\;y_{5}=b,\;y_{6}=c
\end{align*}
if $a\geq(a+b+c)/3\geq b\geq c,$ and
\begin{align*}
x_{1}  &  =x_{2}=\frac{a+b}{2},\;x_{3}=x_{4}=\frac{a+c}{2},\;x_{5}=x_{6}%
=\frac{b+c}{2}\\
y_{1}  &  =a,\;y_{2}=b,\;y_{3}=y_{4}=y_{5}=\frac{a+b+c}{3},\;y_{6}=c
\end{align*}
if $a\geq b\geq(a+b+c)/3\geq c.$ In both cases $\frac{1}{6}\sum_{i=1}%
^{6}\delta_{x_{i}}\prec\frac{1}{6}\sum_{i=1}^{6}\delta_{y_{i}},$ and thus the
inequality $(9)$ follows from Theorem \ref{ThmGHLP}.
\end{proof}

Popoviciu noticed that under the presence of continuity, the inequality $(9)$
works for all triplets $a,b,c\in I$ if and only if the function $f$ is convex.
See \cite{NP2006}, p. 12. Theorem \ref{TPOP} allows this inequality to work
for \emph{certain} triplets $a,b,c$ even when $f$ is not convex. For example,
this is the case of the function $\log^{2}x$, and all points $a,b,c\in
(0,a^{\ast}]$ such that $\frac{a+b}{2},\frac{a+c}{2},\frac{b+c}{2}\in(0,2].$

\begin{remark}
The theory of points of convexity and our generalization of the
Hardy-Littlewood-P\'{o}lya theorem stated in Theorem \ref{ThmGHLP} extend
verbatim to the context of spaces with global nonpositive curvature. See
\cite{NRov2014} for the theory of convex functions on such spaces.
\end{remark}

\section{An Application to Mathematical Finance}

In the context of probability theory, Jensen's inequality is generally stated
in the following form: if $X$ is a random variable and $f$ is a continuous
convex function on an open interval containing the range of $X$, then%
\[
f(E(X))\leq E(f(X)),
\]
provided that both of expectations $E(X)$ and $E(f(X))$ exist and are finite.

A nice illustration of this inequality in mathematical finance refers to the
so called \emph{risk aversion}, the reluctance of someone who wants to invest
his life savings into a stock that may have high expected returns (but also
involves a chance of losing value), preferring to put his or her money into a
bank account with a low but guaranteed interest rate. Indeed, if the utility
function $f$ is concave, then%
\[
f(E(X))\geq E(f(X)).
\]

Using the technique of pushing-forward measures (i.e., of image measures), we
will show that this inequality still works when $f$ is continuous and $E(X)$
is a point of concavity of $f$ relative to its whole domain. This follows from
the following technical result.

\begin{theorem}
\label{thmJintegral}Suppose that $f$ is a real-valued continuous function
defined on an open interval $I$ and $X$ is a random variable associated to a
probability space $\,(\Omega,\Sigma,\mu)$ such that

$(i)$ the range of $X$ is included in the interval $I;$

$(ii)$ the expectations $E(X)$ and $E(f(X))$ exist and are finite;

$(iii)$ $E(X)$ is a point of convexity of $f$ relative to $I.$

Then
\[
f(E(X))\leq\int_{\Omega}f(X(\omega))d\mu(\omega)
\]

\end{theorem}

\begin{proof}
Since $X:\Omega\rightarrow I$ is a $\mu$-integrable map, the push-forward
measure $\nu,$ given by the formula $\nu(A)=\mu(X^{-1}(A)),$ is a Borel
probability measure on $I$ with barycenter $b_{\nu}=\int_{\Omega}X(\omega
)d\mu(\omega)=E(X).$ We have to prove that%
\[
f(b_{\nu})\leq\int_{I}f(x)d\nu(x).
\]
When $\nu$ is a discrete measure, this follows from the fact that $b_{\nu}$ is
a point of convexity. If the range of $X$ is included in a compact subinterval
$K$ of $I$, then the support of $\nu$ is included in $K$ and we have to use
the following approximation argument proved in \cite{NP2006}, Lemma 4.1.10, p.
183: every Borel probability measure $\nu$\ on a compact convex set $K$\ is
the pointwise limit of a net of discrete Borel probability measures
$\nu_{\alpha}$\ on $K$, each having the same barycenter as $\nu.$

In the general case, we approximate $X$ by the sequence of bounded random
variables $X_{n}=\sup\left\{  \inf\left\{  X,n\right\}  ,-n\right\}  .$
\end{proof}

\section{Concluding Remarks}

In this paper we introduced the concept of convexity at a point relative to a
convex subset of the domain. This fact made Jensen's inequality available to a
large variety of nonconvex functions and shed new light on the
Hardy-Littlewood-P\'{o}lya theorem of majorization. In turn, the probabilistic
form of Jensen's inequality (as stated in Theorem \ref{thmJintegral}) put in a
more general perspective the problem of risk aversion.

Most likely the notion of convexity at a point could have a practical purpose
in optimization theory, information theory, the design of communication
systems etc.

\medskip

\noindent\textbf{Acknowledgement}. The first author was supported by a grant
of the Romanian National Authority for Scientific Research, CNCS -- UEFISCDI,
project number PN-II-ID-PCE-2011-3-0257. The second author was supported by
the strategic grant POSDRU/159/1.5/S/133255, Project ID 133255 (2014),
co-financed by the European Social Fund within the Sectorial Operational
Program Human Resources Development 2007 - 2013.

\end{document}